\documentclass[12pt]{amsart}

\usepackage{amsmath,amsfonts,amssymb,amsthm}
\usepackage{pinlabel}
\usepackage{graphicx}
\usepackage{mathtools}
\usepackage[all]{xy}
\usepackage{hyperref}
\usepackage[usenames,dvipsnames]{color}
\usepackage{color}
\xyoption{dvips}

\numberwithin{equation}{section}

\newtheorem{thm}{Theorem}[section]
\newtheorem{cor}[thm]{Corollary}
\newtheorem{question}[thm]{Question}

\newtheorem{prop}[thm]{Proposition}
\newtheorem{lem}[thm]{Lemma}

\theoremstyle{remark}
\newtheorem{rem}[thm]{Remark}

\theoremstyle{definition}
\newtheorem{defn}[thm]{Definition}
\theoremstyle{remark}
\newtheorem{remark}[thm]{Remark}
\theoremstyle{property}

\newcommand{\R}{\mathbb{R}}

\newcommand{\Q}{\mathbb{Q}}
\newcommand{\Qc}{\mathcal{Q}}
\newcommand{\Z}{\mathbb{Z}}
\newcommand{\C}{\mathbb{C}}

\newcommand{\T}{\mathbb{T}}
\newcommand{\F}{\mathbb{F}}

\newcommand{\id}{\mathrm{id}}

\newcommand{\OP}{\operatorname}

\begin{document}

\title{On Legendrian products and twist spuns}

\author{Georgios Dimitroglou Rizell}
\author{Roman Golovko}

\begin{abstract}
The Legendrian product of two Legendrian knots, as defined by Lambert-Cole, is a Legendrian torus. We show that this Legendrian torus is a twist spun whenever one of the Legendrian knot components is sufficiently large. We then study examples of Legendrian products which are not Legendrian isotopic to twist spuns.
In order to do this, we prove a few structural results on the bilinearised Legendrian contact homology and augmentation variety of a twist spun.
In addition, we show that the threefold Bohr--Sommerfeld covers of the Clifford torus and Chekanov torus are not twist spuns.
\end{abstract}

\address{Department of Mathematics, Uppsala University, Box 480, SE-751 06, Uppsala, Sweden}
\email{georgios.dimitroglou@math.uu.se}

\address{Faculty of Mathematics and Physics, Charles University, Sokolovsk\'{a} 83, 18000 Praha 8, Czech Republic}
\email{golovko@karlin.mff.cuni.cz}
\date{\today}



\maketitle

\section{Introduction}

\subsection{Basic notions}
Here we mainly consider Legendrian submanifolds of the standard contact vector space $(\R^{2n+1},\alpha)$ with coordinates $(x_1,y_1,\dots, x_n, y_n,z)$, where $\alpha$ is the standard contact form $\alpha \coloneqq dz-\sum_i y_idx_i$, whose associated {\em Reeb vector field} thus is $\partial_z$. We will also consider general \emph{contactisations}, i.e.~contact manifolds $(P \times \R,\alpha)$ for general $2n$-dimensional Liouville manifold $(P^{2n},d\eta)$ with the contact form $\alpha=dz+\eta.$ A {\em Legendrian submanifold} is an $n$-dimensional smooth submanifold that satisfies $T_x\Lambda \subset \ker \alpha$ for all $x\in \Lambda,$ and a {\em Legendrian isotopy} is a smooth isotopy through Legendrian submanifolds.

Integral trajectories of the Reeb vector field $\partial_z$ which start and end on a Legendrian submanifold $\Lambda$ are called {\em Reeb chords} of $\Lambda$. The set of Reeb chords of $\Lambda$ will be denoted by $\Qc(\Lambda)$, which is a finite set in the case of a generic compact Legendrian. To each $c\in \Qc(\Lambda)$ we associate its length $\ell(c) \coloneqq \int_c \alpha>0$.

\begin{defn}
For Legendrian submanifolds $\Lambda_1$, $\Lambda_2$, we say that:
 \begin{itemize}
\item \emph{$\Lambda_1$ is smaller than $\Lambda_2$} and write $\Lambda_1<\Lambda_2$ if the length of the longest Reeb chord on $\Lambda_1$ is strictly smaller than the length of the shortest Reeb chord on $\Lambda_2$;
\item \emph{$\Lambda_1$ and $\Lambda_2$ have distinct Reeb chord lengths} if no Reeb chord of $\Lambda_1$ has the same length as a Reeb chord of $\Lambda_2.$
\end{itemize}
\end{defn}

Recall that an $n$-dimensional immersion $\iota \colon L \looparrowright (P,d\eta)$ is \emph{exact Lagrangian} if $\eta$ pulls back to an exact form $\iota^*\eta=df$. Any exact Lagrangian gives rise to a Legendrian \emph{immersion} by prescribing the lift $\{z \circ \iota=-f\}.$ Conversely, the {\em Lagrangian projection} is given as $\Pi_P: P \times \R \to P$ and sends a Legendrian submanifold to an exact Lagrangian immersion, whose double points moreover correspond bijectively to the Reeb chords.

In the case of $\R^{2n+1}$ we also recall the definition of the {\em front projection} given by $\Pi_{\OP{Fr}}: \R^{2n+1} \to \R^{n+1}$, $\Pi_{\OP{Fr}}(\mathbf{x},\mathbf{y},z)=(\mathbf{x},z)$. A Legendrian submanifold $\Lambda$ can be recovered from its front projection.

We are here interested in comparing two well-known geometric constructions that can be used to produce Legendrian submanifolds of higher dimension from lower dimensional ones.

\subsection{The Legendrian product}
\label{sec:LegPr}
\begin{defn}
 Consider two Legendrian submanifolds $$\iota_i \colon \Lambda_i \hookrightarrow (P_i \times \R,dz_i+\eta_i),\ i=1,2.$$
The {\em Legendrian product}
$\Lambda_1\boxtimes \Lambda_2 \looparrowright (P_1 \times P_2 \times \R,dz+\eta_1+\eta_2)$ is the Legendrian immersion defined by
$$\iota_1 \boxtimes \iota_2(u_1,u_2) = (\Pi_{P_1}(\iota_1(u_1)),\Pi_{P_2}(\iota_1(u_1)),z_1(\iota_1(u_1))+z_2(\iota_2(u_2)))$$
\end{defn}
The Legendrian product is \emph{embedded} if and only if the two Legendrians have distinct Reeb chord lengths. Legendrian products were introduced and studied in \cite{LegendrianProducts} by Lambert-Cole who, among other things, computed their classical invariants. Also see the work \cite{GeneratingFamilyLC} by the same author for results about existence of generating families for products as well as computations of certain Morse flow-trees.

\begin{remark}
The Legendrian isotopy class of $\Lambda_1\boxtimes \Lambda_2$ is only invariant under Legendrian isotopy of the two factors as long as the pair has distinct Reeb chord lengths for each time in the isotopy; in particular, the Legendrian product thus depends on the choice of the Legendrian \emph{embeddings} $\Lambda_i\hookrightarrow (P_i\times \R, dz_i+\eta_i)$ and not merely on their Legendrian isotopy classes.
\end{remark}
 Assume that $\Lambda_1 < \Lambda_2.$ In this case we can rescale the Reeb chords on $\Lambda_2$ to make them even longer by applying  the contact isotopy
\begin{gather*}
(P_2 \times \R,dz_2+\eta_2) \to (P_2 \times \R,e^{-t}(dz_2+\eta_2)),\\
(p,z_2) \mapsto (\varphi^t_{\eta_2}(p),e^tz_2),
\end{gather*}
induced by the lift of the Liouville flow. This has the effect of rescaling the length of each Reeb chord of $\Lambda_2$ by $e^t$; the Reeb chord lengths clearly remain distinct from those of $\Lambda_1$ throughout the isotopy. This rescaling  will be used repeatedly in constructions below.

The Chekanov--Eliashberg algebra of a Legendrian product is expected to contain more information than merely what is contained in the DGAs of the knots themselves \cite{LegendrianProducts}. For instance, pseudoholomorphic discs that have boundary on each knot and more than one positive puncture naturally enter the computation of the differential. Due to the transversality issues that arise in such considerations, we still lack a structural understanding of the DGAs of Legendrian products. For the same reason we have not been able to answer the following natural question:
\begin{question}
Is there a Legendrian torus in $\R^{5}$ which is not obtained as a Legendrian product?
\end{question}
We expect the answer to be negative; more precisely, we believe that the Legendrian tori from \cite{LegendrianBSLagrangianTori} that we also consider in Section \ref{sec:other} below are not products.

\subsection{The twist spun}
\label{sec:twistspun}
 Twist spuns are examples of Legendrian tori that have received somewhat more attention than Legendrian products.  Given a loop $\{\Lambda_\theta\},$ $\theta \in S^1,$ of Legendrian embeddings of $\Lambda$ in $(P \times \R,dz+\eta)$, the corresponding mapping torus has a natural Legendrian embedding $$\Sigma_{S^1}\{\Lambda_\theta\} \subset (\R^{2}\times P\times \R,dz-ydx+\eta)$$
 called the \emph{twist spun}  that was first constructed and studied by Ekholm--K\'{a}lm\'{a}n in \cite{IsotopiesTori}. There are also higher dimensional versions, which we recall in the general construction below.  In the special case when the loop of Legendrians is constant, we recover the so-called $S^k$-spun of Legendrians which first appeared in \cite{NonisotopicLegendrianSubmanifolds} for $k=1$ and later was generalised and studied in the case of all $k \ge 1$ by the second author in \cite{FrontSpinningConstruction}.

{ 
Let $\Lambda_{\theta}\subset (P \times \R, dz+\eta)$, $\theta \in S^k$, $k\geq 1$, be a $S^k$-family of Legendrian embeddings.
\begin{defn}
\label{def:susp}
The \emph{suspension} of the family $\Lambda_\theta \subset P \times \R$ is the unique Legendrian
$$\Sigma_{S^k}\{\Lambda_\theta\} \subset (P \times T^*S^k \times \R,dz+\eta-ydx)$$
determined by the property that its image under the canonical projection $$\Pi: P \times T^*S^k \times \R \to P \times S^k \times \R$$
is equal to
$$ \{ (x,\theta,z) \in P \times S^k \times \R;\:\: (x,z) \in \Lambda_\theta \},$$
which is an embedded submanifold.
\end{defn}

\begin{rem}
 The Legendrian condition $$(dz+\eta-ydx)|_{T\Sigma_{S^k}\{\Lambda_\theta\}}=0$$ implies that the momentum coordinate of $\Sigma_{S^k}\{\Lambda_\theta\}$ (the $y$--coordinate in the case $k=1$), and thus the Legendrian submanifold itself, can be uniquely recovered from its projection $\Pi(\Sigma_{S^k}\{\Lambda_\theta\})$.
\end{rem}

\begin{rem}
 
The suspension $\Sigma_{S^k}\{\Lambda_\theta\}$ is diffeomorphic to the mapping torus which is induced by the family of embeddings; this is the total space of a fibre bundle over $S^k$ with fibres diffeomorphic to $\Lambda_\theta$.
\end{rem}
There is a canonical exact symplectic inclusion
$$T^*(\R \times S^k) \subset T^*(\R^{k+1})$$
induced by the smooth identification $$\R \times S^k \cong \R^{k+1}\setminus \{0\}\subset \R^{k+1}$$
where $\{t\} \times S^k$ is mapped to the sphere of radius $e^t$. In the case when $P$ is a subcritical Liouville domain, i.e.~when there exists a choice of splitting
$$P=Q\times \R^{2}=Q \times T^*\R,$$ we can use this inclusion to construct a strict contact embedding
$$ \label{neighbPxcodb}
P \times T^*S^k \times \R = Q\times T^*(\R \times S^k)\times \R \hookrightarrow Q\times T^*(\R^{k+1}) \times \R = (P \times \R^{2k} \times \R,dz+\eta-y_idx_i) $$
of the corresponding contactisations. In the first equality we have made use of the canonical identification
$$T^*\R \times T^*S^k=T^*(\R \times S^k)$$
while
$$T^*\R^{k+1}=T^*\R \times T^*\R^k=T^*\R \times \R^{2k}$$
was used in the last equality. 
\begin{defn}
The image of the Legendrian suspension $\Sigma_{S^k}\{\Lambda_\theta\} \subset P \times \R^{2k+1}$ under the canonical strict contact embedding \eqref{neighbPxcodb} is called the {\em $S^k$-twist spun of the family $\{\Lambda_{\theta}\}_{\theta \in S^k}$}.
\end{defn}
From now on we call the $S^k$-twist spun of $\{\Lambda_{\theta}\}$, $\theta\in S^k$, simply the {\em twist spun of $\Lambda_{\theta}$} and in order to simplify the notations, when no confusion can arise, we omit $S^k$ in the notation and simply write $\Sigma\{\Lambda_\theta\}$.

\begin{remark}
The twist spinning construction can be seen as a version of the Lagrangian bundle construction due to Audin--Lalonde--Polterovich \cite{Audin-Lalonde-Polterovich}; namely, it embeds the suspension inside a neighbourhood of a subcritical isotropic sphere. In Section \ref{sec:proofs} we give an alternative construction of the twist spun for which this relationship becomes more apparent. This construction also has the advantage that it can be applied to general (not necessarily subcritical) exact symplectic manifolds $P$.
\end{remark}

For most of the results in this paper we will only need the case when $P$ is the standard symplectic vector space,  i.e.~the subcritical Liouville domain
$$ P=\R^{2n}=\R^{2(n-1)}\times T^*\R.$$
In this case the $S^1$-twist spun (i.e.~$k=1$)  has the following more direct description. 
\begin{lem}
The twist spun $\Sigma_{S^1}\{\Lambda_\theta\}$ inside $\{(\tilde{\mathbf{x}},\tilde{\mathbf{y}},\tilde{z})\} = \R^{2(n+1)+1}$ can be explicitly expressed as follows. Given the parametrisation $(\mathbf{x}(\theta,q),\mathbf{y}(\theta,q),z(\theta,q)) \in \R^{2n+1}$ in locally defined coordinates on the mapping torus, we can write
$$
\begin{cases}
 & \tilde{\mathbf{x}}=(x_1(\theta,u)\cos{\theta}, x_1(\theta,u)\sin{\theta},x_2(\theta,u),\ldots,x_{n}(\theta,u)),\\
 & \tilde{\mathbf{y}}=(y_1(\theta,u)\cos{\theta}-\partial_\theta z(\theta,u)\sin{\theta}, y_1(\theta,u)\sin{\theta}+\partial_\theta z(\theta,u)\cos{\theta},y_2(\theta,u),\ldots,y_{n}(\theta,u)),\\
 & \tilde{z}=z(\theta,u).
\end{cases}
$$
\end{lem}

One very useful property of the above construction is that the Chekanov--Eliashberg algebra of the resulting Legendrian can be explicitly determined in terms of the Chekanov--Eliashberg algebra of the original one by a type of K\"{u}nneth formula; see \cite{IsotopiesTori} which we recall in Section \ref{sec:spundga} below, as well as the partial results in high dimension by the authors in \cite{EstimatingReebChordsCharacteristicAlgebra}. (In fact, in the case of $S^1$-spuns, the DGA is closely related to the so-called Baues--Lemaire cylinder object in the category of DG-algebras from \cite{Models}; it can be interpreted as the $S^1$-analogue.)

\subsection{Results}

 Let $W^k \subset \R^{2k+1}$ be the standard Legendrian $k$-sphere which has a single transverse Reeb chord, and whose front projection is the rotation symmetric ``flying saucer'' with a single cusp edge; see Figure \ref{whitneyimmersion} for the cases $k=1,2$.  It was shown by Lambert-Cole in \cite[Corollary 1.5] {LegendrianProducts} that the Legendrian product $\Lambda\boxtimes W^k$ is Legendrian isotopic to the front $S^k$-spun of $\Lambda$ whenever $\Lambda<W^{k}$ is satisfied. One of our goals is to extend Lambert-Cole's result by showing that it remains true also if $W^k$ is replaced by Legendrian spheres of a certain more general type. In order to do this,  we are in need of the following definition:
\begin{defn}
We say that two Legendrians $\Lambda,\Lambda' \subset (Y,\alpha)$ of a contact manifold are $n$-stably formally Legendrian isotopic if their stabilisations
$$\{0\} \times \Lambda, \{0\} \times \Lambda' \subset (\R^{2n}\times Y,-ydx+\alpha)$$
are formally isotopic as subcritical isotropic manifolds.
\end{defn}
\begin{remark}\begin{enumerate}
\item  Gromov's $h$-principle for subcritical isotropic embeddings \cite[12.4.1]{hprinciple} together with the isotropic normal neighbourhood theorem \cite{Geiges} readily imply that the following is satisfied for two Legendrians which are $n$-stably formally Legendrian isotopic: After stabilising each Legendrian by a suitable family of rotating Legendrian $n$-planes (the rotations depend on the comparison of framings of the symplectic normal bundles of the two isotropic embeddings), the open Legendrians obtained are Legendrian isotopic. This is exploited in the proof of Theorem \ref{thm:normalbundle} in Section \ref{sec:proofs1} below.
\item Assume that $n \ge k> 0$.
Since $\pi_k (U(k+n)/U(n))=0$ is satisfied in this case, it readily follows that all Legendrian $k$-dimensional spheres are $n$-stably formally Legendrian isotopic. In particular, all Legendrian knots are $n$-stably formally Legendrian isotopic when $n \ge 1$.
\item The property of being $n$-stably formally Legendrian isotopic automatically implies being $n+1$-stably formally Legendrian isotopic. In particular, $0$-stably formally Legendrian isotopic (which is the same as formally Legendrian isotopic) implies $n$-stably formally Legendrian isotopic for all $n \ge 0$. The converse, however, does not hold in general. 
\end{enumerate}
\end{remark}

The result of Lambert-Cole admits the following generalisation:

\begin{thm}
 
\label{thm:normalbundle}
Let $\Lambda_1 \subset (\R^{2n_1+1}, \alpha)$ be a Legendrian submanifold and $\Lambda_2 \subset (\R^{2n_2+1}, \alpha)$ a Legendrian sphere which is $n_1$-stably formally Legendrian isotopic to the standard sphere $W^{n_2}$. If $\Lambda_1 < \Lambda_2,$ then it follows that $\Lambda_1 \boxtimes \Lambda_2$ is Legendrian isotopic to a twist spun of an $S^{n_2}$-family of Legendrians which is obtained from $\Lambda_1$ by a suitable family of rotations, each being a contact lift of the linear symplectic $U(n_1)$-action. Furthermore: 
\begin{itemize}
 
\item In the case when $\Lambda_2$ is formally Legendrian isotopic to the standard sphere $W^k$, the product is Legendrian isotopic to the ordinary spun of $\Lambda_1$, i.e.~the twist-spun $\Sigma\{\Lambda_1\}$ of the constant family. 
\item  In the case when $n_2=1$ it follows that $\Lambda_1 \boxtimes \Lambda_2$ is Legendrian isotopic to the twist spun of the family of Legendrians that covers the image of $\Pi_{\R^{2n_1}}(\Lambda_1)$ under the family
$$ (z_1,z_2,\ldots,z_{n_1}) \mapsto (e^{i\OP{rot}{(\Lambda_2)}\cdot\theta}z_1,z_2,\ldots,z_{n_1}), \:\: \theta \in S^1,$$
of rotations.
\end{itemize}
\end{thm}

Given an exact Lagrangian cobordism $ L \subset (\R \times P_1 \times \R,d(e^t(dz+\eta_1)))$ from $\Lambda^-_1$ to $\Lambda^+_1 \subset P_1 \times \R$ (see Section \ref{sec:cob} below) we can again form a product
$$L\boxtimes \Lambda_2 \looparrowright (\R \times P_1 \times P_2 \times \R, d(e^t(dz+\eta_1+\eta_2))),$$
which similarly to the Legendrian product we define to be the Lagrangian immersed cobordism
$$\iota_1 \boxtimes \iota_2(u_1,u_2)=(\Pi_{\R \times P_1}(\iota_1(u_1)),\Pi_{P_2}(\iota_2(u_2)),z_1(\iota_1(u_1))+z_2(\iota_2(u_2)))$$
from $\Lambda^-_1 \boxtimes \Lambda_2$ to $\Lambda^+_1 \boxtimes \Lambda_2.$ Note that one still can speak about Reeb chords on $L,$ by which we mean an integral curve of the Reeb vector field $\partial_z$ that is contained inside some slice $\{t\} \times P_1 \times \R$ and which has endpoints on $L$. In other words, the condition that $L$ and $\Lambda_2$ have distinct Reeb chord lengths still makes sense.
\begin{thm}
\label{thm:filproductlooseproducttwist}
\begin{itemize}
\item[(i)] The cobordism $L \boxtimes \Lambda_2$ is an immersed exact Lagrangian cobordism from $\Lambda^-_1 \boxtimes \Lambda_2$ to $\Lambda^+_1 \boxtimes \Lambda_2$ which is embedded when the Reeb chord lengths of $L$ are distinct from those of $\Lambda_2$; and
\item[(ii)] If $\Lambda_1<\Lambda_2$ and $\Lambda_1$ is stabilised or loose, then $\Lambda_1\boxtimes \Lambda_2$ is loose as well.
\end{itemize}
\end{thm}
\begin{remark}
 
In the case when $\Lambda^{\pm}_1<\Lambda_2$ we can always arrange so that the assumptions in Part (i) of the above theorem are met. Namely, we can then perform an initial rescaling of $\Lambda_2$ by applying the contact isotopy $(\varphi^{t}_{\eta_2},e^{t}\cdot)$ induced by the positive Liouville flow on $(P_2,d(\eta_2))$. Since $\Lambda^{\pm}_1<\Lambda_2$, we obtain Legendrian isotopies of the two products $\Lambda^{\pm}_1 \boxtimes \Lambda_2$ and, after a sufficiently large such rescaling, the Reeb chords on $\Lambda_2$ all become greater than the chords on the cobordism.
\end{remark}

The main aim of this note is to give an example of a Legendrian product which is \emph{not} Legendrian isotopic to a twist spun:
\begin{thm}
\label{thm:productnotspunnottwistspun}
There exist Legendrian product tori $\Lambda=\Lambda_1 \boxtimes \Lambda_2\subset \R^5$ which are not Legendrian isotopic to any twist spun of a family of Legendrian knots.
\end{thm}
For instance $\Lambda_1$ can be taken to be the standard unknot shown to the left in Figure \ref{fig:NonspunProduct} while $\Lambda_2$ is the particular version of the stabilised unknot shown in the same figure to the right.

In addition, in Section \ref{sec:other} we prove that the threefold Bohr--Sommerfeld covers of the Clifford torus and Chekanov torus discussed by the authors in \cite{LegendrianBSLagrangianTori} are not twist spuns. Finally, we conjecture that no threefold Bohr--Sommerfeld cover of one of Vianna's monotone Lagrangian tori in $\C P^2$ is a twist spun; this infinite family of Lagrangian tori in pairwise distinct Hamiltonian isotopy classes was constructed in \cite{Vianna16}.

\section*{Acknowledgements}
This project started when the second author attended the 39th Winter School on Geometry and Physics in Srni, and we are grateful to the organisers of the workshop for their hospitality.
In addition, we would like to thank Peter Lambert-Cole and Dmitry Tonkonog for useful discussions.
The first author is supported by the grant KAW 2016.0198 from the Knut and Alice Wallenberg Foundation.
The second author is supported by the GA\v{C}R EXPRO Grant 19-28628X.

\section{Background}
Here we consider some basic notions from the theory of Legendrian submanifolds. In particular, we introduce two important and distinct classes of Legendrian submanifolds: loose Legendrian submanifolds and fillable Legendrian submanifolds. These two classes have very different properties. Loose Legendrian submanifolds belong to ``flexible contact topology.'' More precisely, they satisfy an h-principle due to Murphy \cite{LooseLegendrianEmbeddings}. Fillable Legendrian submanifolds of contactisations belong to ``rigid contact topology''; the Chekanov--Eliashberg algebra is a powerful Legendrian isotopy invariant that has been shown to be an efficient tool for distinguishing such Legendrians.  We end the section by recalling the necessary background on this invariant.

\subsection{Loose Legendrians}
The class of {\em loose Legendrian} submanifolds were introduced by Murphy in \cite{LooseLegendrianEmbeddings}, where they were shown to satisfy an $h$-principle. In particular, they are classified up to Legendrian isotopy by their topological properties. We now recall the definition.

\begin{figure}[h!]
\begin{center}
\labellist
\pinlabel $z$ at 128 540
\pinlabel $a$ at 135 520
\pinlabel $1$ at 185 480
\pinlabel $-1$ at 65 480
\pinlabel $x$ at 205 488
\pinlabel $-a$ at 135 457
\endlabellist
\includegraphics[height=4.5cm]{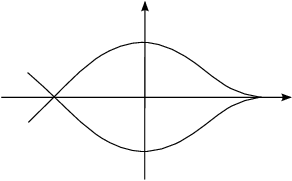}
\caption{The front projection of
$\gamma$.}\label{looseproj}
\end{center}
\end{figure}
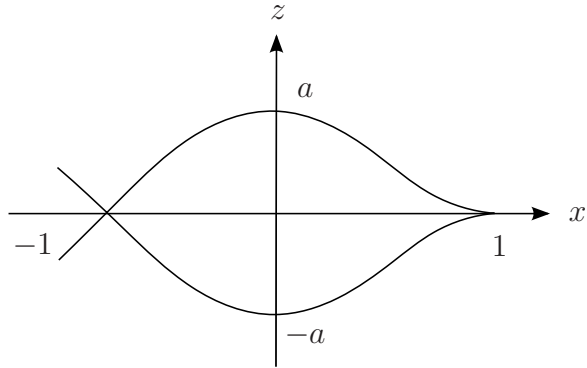

We say that a Legendrian submanifold $\Lambda\subset (P^{2n} \times \R, \alpha)$, $n\geq 2$, is {\emph loose} if there exists a pair of neighbourhoods $(U,\Lambda_0) \subset (P\times \R,\Lambda)$ that admits a contactomorphism to the {\em standard loose chart} $(R_{abc}, \Lambda_{0})$ with $a<bc$.
Here $R_{abc} \subset (\R^{2n+1},dz-y_idx_i)$ is a standard Darboux neighbourhood defined by
\begin{align*}
R_{abc} =& \{ (x,y, x_{1},\dots, y_{n-1}, z);\:\: |x|,|y|\leq 1, \|(x_{1},\dots,x_{n-1})\|\leq b,\\ &\|(y_{1},\dots,y_{n-1})\|\leq c, |z|\leq a )\}\subset (\R^{2n+1}, \alpha)
\end{align*}
and $\Lambda_{0}$ is the Legendrian solid cylinder, which is the product of
\begin{align*}
D_{b} = \{ (x_{1},y_{1},\dots,x_{n-1},y_{n-1}); \:\: y_{1} = \dots = y_{n-1} = 0, \|(x_{1},\dots,x_{n-1})\|\leq b \}
\end{align*}
and a Legendrian curve $\gamma
\subset \R^{3}$ with coordinates $(x,y,z)$ and whose front projection is described in Figure~\ref{looseproj}; in particular, the Legendrian arc $\gamma$ is contained inside the box
\begin{align*}
Q_{a}=\{|x|\leq 1, |y|\leq 1, |z|\leq a\}
\end{align*}
with $\partial \gamma\subset \partial Q_{a}$.

\subsection{Lagrangian cobordisms and fillings}
\label{sec:cob}

An \emph{exact Lagrangian cobordism from $\Lambda^-$ to $\Lambda^+ \subset P \times \R$} is a properly embedded $(n+1)$-dimensional submanifold
$$ L \subset (\R \times P^{2n} \times \R,d(e^t(dz+\eta)))$$
which
\begin{itemize}
\item coincides with a cylinder over $\Lambda^+$ inside $[T,+\infty) \times P \times \R$,
\item coincides with a cylinder over $\Lambda^-$ inside $(-\infty,-T] \times P \times \R$, and
\item is exact Lagrangian in the sense that $e^t(dz+\eta)|_{TL}$ is exact with a globally constant primitive on $(-\infty,-T] \times \Lambda^- \subset L.$
\end{itemize}
We also allow the case when $\Lambda^{-}=\emptyset$; if this holds then we call $L$ an {\em exact Lagrangian filling of $\Lambda^{+}$}, and we say that $\Lambda^{+}$ is {\em (exact) fillable}.

\subsection{The Chekanov--Eliashberg algebra}

The Chekanov--Eliashberg algebra is a Legendrian invariant introduced by Chekanov \cite{DiffAlg} and Eliashberg \cite{Eliash}, and is a part of the Symplectic Field Theory \cite{IntroSFT} by Eliashberg--Hofer--Givental. The version that we use for Legendrians of contactisations of Liouville domains is due to Ekholm--Etnyre--Sullivan \cite{LegendrianContactPxR}. We proceed to sketch the definition, and refer to the latter article for more details.

The Chekanov--Eliashberg algebra of a closed Legendrian $\Lambda \subset P \times \R$ is a noncommutative semifree DGA $(\mathcal{A}(\Lambda),\partial)$ generated by the Reeb chords on $\Lambda$ over the group ring $\F[H_1(\Lambda)]$. In this article we may restrict attention to $\F=\Z_2,$ but when $\Lambda$ is spin we can also take e.g.~$\F=\Q$ or $\C.$ In the case of Legendrians diffeomorphic to $(S^1)^k$, which is our main interest here, we get an identification of $\F[H_1(\Lambda)]$ with the Laurent polynomial ring in $k$ variables over $\F$. 

The degree of a Reeb chord generator $c$ is determined by the so-called Conley--Zehnder index via $|c|=\OP{CZ}(c)-1.$ The differential $\partial$ satisfies the Leibniz rule $$\partial(\mathbf{a}\mathbf{b})=\partial(\mathbf{a})\mathbf{b}+(-1)^{|a|}\mathbf{a}\partial(\mathbf{b})$$ and is defined on the generators by a count of pseudoholomorphic polygons in $P$ with boundary on $\Pi_P(\Lambda),$ a single positive puncture at the input chord, and negative punctures at the output chords.

The Legendrian invariance comes from the fact that the stable-tame isomorphism type of the Chekanov--Eliashberg algebra is independent of the choice of almost complex structure and invariant under Legendrian isotopy. Here we will only need to be concerned with a slightly weaker notion, which is that of DG-homotopy; see e.g.~\cite[Lemma 3.14]{Ekhoka} for the definition.

An \emph{augmentation} is a unital DGA-morphism $\varepsilon \colon \mathcal{A}(\Lambda) \to \F,$ where the latter field is considered as a unital DGA with an empty set of generators. An augmentation is said to be \emph{graded} if all generators in degrees different from zero are mapped to zero by $\varepsilon$.

Not all Legendrians have Chekanov--Eliashberg algebras that admit augmentations; for instance in the Chekanov--Eliashberg algebra of a loose Legendrian the unit $1$ is a boundary, so it admits no augmentations. On the contrary, in accordance with the principles of symplectic field theory, exact Lagrangian fillings induce augmentations:
\begin{thm}[\cite{IntroSFT,RationalSFT,RationalSFT2}]
\label{thm:filling}
An exact Lagrangian filling $L$ of $\Lambda \subset P \times \R$ induces an augmentation $\varepsilon_L \colon \mathcal{A}(\Lambda) \to \Z_2$ of its Chekanov--Eliashberg algebra with coefficients in $\Z_2.$ If the filling is spin, then one obtains an augmentation with arbitrary coefficient. If the Maslov class of $L$ vanishes, then the augmentation is moreover graded.
\end{thm}

Given an augmentation one can perform Chekanov's linearisation procedure \cite{DiffAlg} to obtain a complex $LCC^\varepsilon_*(\Lambda)$ which is an $\F$-vector spaces spanned by the Reeb chords. This was generalised in \cite{Bilinearised} by Bourgeois--Chantraine to the bilinearised complex $LCC^{\varepsilon_1,\varepsilon_2}_*(\Lambda)$ induced by two augmentations (in fact, they showed that the augmentations form an $A_\infty$-category). Observe that graded augmentations must be used if we want have a well-defined $\Z$-grading of the bilinearised complex.

The invariance of the bilinearised complex under Legendrian isotopy is more complicated to state than the invariance of the DGA, since it depends on the augmentations. Here follows the main invariance result that we need:
\begin{thm}[Corollary 5.12 \& Proposition 5.17 \cite{AugmentationsSheaves}]
\label{thm:dghomotopy}
The quasi-isomorphism class of the homology
$$LCH^{\varepsilon_1,\varepsilon_2}_*(\Lambda)\coloneqq H(LCC^{\varepsilon_1,\varepsilon_2}_*(\Lambda))$$
of the bilinearised Legendrian complex for a Legendrian knot $\Lambda \subset \R^3$ does not depend on the DG-homotopy classes of the involved augmentations $\varepsilon_i$, $i=1,2$.
\end{thm}

\section{Proofs of Theorems \ref{thm:normalbundle} and \ref{thm:filproductlooseproducttwist}}
\label{sec:proofs}

 Before showing the equivalence between twist spuns and products, we will give a slightly different realisation of the $S^k$-twist spun, $k\geq 1$. This presentation has the advantage that it can be performed in arbitrary contact manifolds and, in addition, that it exhibits the relation to Legendrian products more clearly.

For any fixed parametrised Legendrian $k$-sphere
$$f \colon S^k \to S \subset \R^{2k+1}$$
the subset
$$ P \times S \subset (P \times \R^{2k} \times \R,dz+\eta-ydx)$$
has a standard contact neighbourhood which can be identified with
\begin{align}
\label{neighPxcodb}
\psi_S \colon (P \times D_\epsilon T^*S^k \times [-\epsilon,\epsilon],dz+\eta-\eta_{S^k}) \hookrightarrow (P \times \R^{2k} \times \R,dz+\eta-y_idx_i)
\end{align}
by a contact-form preserving embedding that takes a point $(x,q) \in P \times 0_{S^k}$ to $(x,f(q)) \in P \times S.$ Here $\eta_{S^k}$ is the tautological one-form on some radius-$\epsilon$ codisc bundle $D_\epsilon T^*S^k$. See e.g.~\cite{Geiges} for a treatment of the standard contact neighbourhood theorem.

Assume that we are given a smooth $S^k$-family of Legendrians $\{\Lambda_\theta\}$ in $P$. After applying the contactomorphism $(\varphi^{-t}_{\eta},e^{-t}\cdot)$ induced by the negative Liouville flow in $P$ to the $\Lambda_\theta$, we may assume that the Legendrian suspension $\Sigma_{S^k}\{\Lambda_\theta\} \subset P \times T^*S^k\times \R$ (see Definition \ref{def:susp}) is contained inside the domain $P \times D_\epsilon T^*S^k \times [-\epsilon,\epsilon]$ of the contact embedding $\psi_S$ described in (\ref{neighPxcodb}) above. To see this, note that the aforementioned rescaling shrinks both the momentum coordinate in $T^*S^k$ and the $z$--coordinate of the trace. 

We consider the image of the suspension under the same embedding and thus obtain a different realisation of the twist spun inside $P \times \R^{2k} \times \R$ which we denote by $\Sigma_{\psi_S}\{\Lambda_\theta\}$. The following lemma is immediate.
\begin{lem}
\label{lem:spunproduct}
 Assume that $S>\Lambda$ is satisfied for a Legendrian sphere $S \subset \R^{2k+1}$ and a Legendrian submanifold $\Lambda \subset P \times \R$. After shrinking $\Lambda$ sufficiently by an application of the negative Liouville flow in $P$ as described above the twist spun $\Sigma_{\psi_S}\{\Lambda\} \subset P \times \R^{2k+1}$ of the constant family $\{\Lambda\}$ is well-defined, and moreover equal to the Legendrian product $\Lambda \boxtimes S \subset P \times \R^{2k+1}$. (Recall that $S>\Lambda$ implies that $\Lambda$ can be shrunk indefinitely without changing the Legendrian isotopy class of the product.)
\end{lem}

An important particular case is when the Legendrian sphere $S$ is taken to be the $k$-dimensional standard Legendrian sphere $W^k \subset \R^{2k+1}$ with a unique Reeb chord, which e.g.~can be constructed as the Legendrian lift of the Whitney immersion in $\R^{2k}$ (this is sometimes also called the Whitney sphere). When $k=1$, the Lagrangian projection is a self-transverse figure-8 curve that bounds a total of zero area; in general, the front projection is given by the rotationally symmetric ``flying saucer'' shown in Figure \ref{whitneyimmersion} in the case $k=2$.

\begin{figure}[h!]
\vspace{5mm}
\begin{center}
\labellist
\pinlabel $\Pi_{\OP{Fr}}(W^2)$ at 150 5
\pinlabel $\Pi_{\R^2}(W^1)$ at -15 5
\endlabellist
\includegraphics[height=3.2cm]{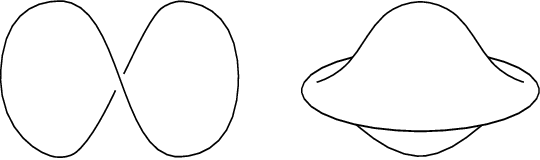}
\caption{The Lagrangian projection of $W^1$ (left) and the front projection of $W^2$ (right).}
\label{whitneyimmersion}
\end{center}
\end{figure}

 The following lemma allows us to to compare the version of the twist spun described in Section \ref{sec:twistspun} and the above version of the twist spun in the case $S=W^k$.
\begin{lem}
\label{lem:whitneyiso}
In the case $P=\R^{2n}$, the inclusion $\psi_{W^k}$ restricted to a sufficiently small neighbourhood of
$$\{0\} \times 0_{S^k} \times \{0\} \subset \R^{2n} \times T^*S^k \times \R$$
is contact isotopic to the corresponding restriction of the contact inclusion given by (\ref{neighbPxcodb}).
\end{lem}
\begin{proof}
The statement is a consequence of the fact that the image of the subcritical isotropic sphere $\{0\} \times 0_{S^k} \times \{0\}$ under $\psi_{W^k}$ and the image of the same under the inclusion (\ref{neighbPxcodb}) are related by an ambient contact isotopy. After such a contact isotopy, the two contact embeddings can thus been made to coincide on the subset $\{0\} \times 0_{S^k} \times \{0\}$. Since the contact isotopy that relates the isotropic spheres moreover can be taken to preserve the canonical framings of their symplectic normal bundles, a standard argument then finally allows us to isotope one contact embedding to the other by deforming it in some small neighbourhood of the same subset.

The needed subcritical isotopy can be explicitly constructed, using the
fact that both subcritical isotropic submanifolds arise as
subsets of the standard Legendrian sphere $W^{n+k} \subset \R^{2(n+k)+1}$. More
precisely, the subcritical embedding $\psi_{W^k}(\{0\} \times
0_{S^k} \times \{0\})$ has image given by the intersection
$$W^{n+k} \cap (\{0\} \times \R^{2k} \times \R)=\{0\} \times W^k$$
while the isotropic image of the same sphere under (\ref{neighbPxcodb}) is equal to the cusp-edge of the intersection
$$W^{n+k} \cap (\{0\} \times \R^{2(k+1)+1})=\{0\} \times W^{k+1}.$$
There clearly exists an isotopy through isotropic embeddings that relates these two spheres, which can be realised by a suitable family of embeddings
$$S^k \hookrightarrow W^{n+k} \subset \R^{2(n+k)+1}$$
with images contained inside the higher-dimensional Legendrian sphere $W^{n+k}$.
\end{proof}

\subsection{Proof of Theorem \ref{thm:normalbundle}}
\label{sec:proofs1}

By assumption, the subcritical isotropic submanifold $\{0\} \times \Lambda_2 \subset \R^{2n_1} \times \R^{2n_2+1}$ is formally isotopic to the standard sphere $\{0\} \times W^{n_2}$. We can now invoke the h-principle \cite[12.4.1]{hprinciple} due to Gromov, by which there exists a contact isotopy that takes the first isotropic embedding to the second. (This is the main technical ingredient of the proof.) As a consequence of this contact isotopy, by using the standard neighbourhood theorem from \cite[Theorem 2.5.8]{Geiges}, it readily follows that $\psi_{\Lambda_2}$ restricted to a sufficiently small neighbourhood of $\{0\} \times \Lambda_2$ is contact isotopic to $\psi_{W^{n_2}} \circ R$
for a contactomorphism
\begin{gather*}
\tilde{R} \colon \R^{2n_1} \times T^*S^{n_2} \times \R \to \R^{2n_1} \times T^*S^{n_2} \times \R,\\
((\mathbf{x},\mathbf{y}),(\mathbf{q},\mathbf{p}),z) \mapsto (R_{\mathbf{q}}(\mathbf{x},\mathbf{y}),(\mathbf{q},\mathbf{p}+H(\mathbf{x},\mathbf{y},\mathbf{q})),z+G(\mathbf{x},\mathbf{y},\mathbf{q})), \:\: R_{\mathbf{q}} \in U(n_1),
\end{gather*}
that lifts the symplectic suspension of a suitable family of unitary rotations. Here  $H(\mathbf{x},\mathbf{y},\mathbf{q})$ is Hamiltonian that generates the family $R_{\mathbf{q}}$ of linear rotations, uniquely determined by the requirement $H(0,0,\mathbf{q}) \equiv 0$, while $G(\mathbf{x},\mathbf{y},\mathbf{q})$ is the function that gives the contact lift to the contactisation, uniquely determined by $G(0,0,\mathbf{q}) \equiv 0$.

The reason why we in general cannot take $R_{\mathbf{q}} \equiv \id_{\R^{2n_1}}$ above, is that the contact isotopy that relates two isotropic submanifolds might not respect the relevant trivialisations of their symplectic normal bundles. The family of rotations is needed to correct this.

It now follows from Lemmas \ref{lem:spunproduct} and \ref{lem:whitneyiso} that $\Lambda_1 \boxtimes \Lambda_2$ is Legendrian isotopic to a twist spun as defined in Section \ref{sec:twistspun}. To that end, it might first be necessary to apply a rescaling of $\Lambda_1$ in order to confine it to a sufficiently small neighbourhood of $\{0\} \in \R^{2n_1+1}$.

\emph{Part (i):} When the Legendrian $\Lambda_2$ itself is formally isotopic to $W^{n_2}$ (and not merely stably formally Legendrian isotopic), we may assume that the canonical framings of the symplectic normal bundles
$$ \R^{2n_1} \times 0_{S^{n_2}} \to 0_{S^{n_2}}$$
are preserved by the initial ambient contact isotopy. In other words, the rotations can be taken to be constantly equal to the identity $R_{\mathbf{q}} \equiv \id_{\R^{2n_1}}$.

\emph{Part (ii):} What remains is to investigate how the Hamiltonian isotopy acts on the framing of the symplectic normal bundle of the isotropic embeddings in the case $n_2=1$. To that end it suffices to compare the Maslov index of $W^{n_2}=W^1$ with that of $\Lambda_2$; the former vanishes while the latter is equal to $2\cdot \OP{rot}(\Lambda_2)$ by assumption. The Hamiltonian isotopy that takes the subcritical isotropic submanifolds to each other thus compensates the difference in Maslov classes by an additional twisting of the trivialisation of the symplectic normal bundle. It follows that $R_{\theta}$ can be taken to be the family of rotations
$$(z_1,z_2,\ldots,z_{n_1}) \mapsto (e^{i\OP{rot}(\Lambda_2)\cdot\theta}z_1,z_2,\ldots,z_{n_1}), \:\: \theta \in S^1,$$
in $U(n_1)$ as sought.  There is an ambiguity of the rotation number $\OP{rot}(\Lambda_2) \in \Z$; its sign depends on the choice of an orientation of $\Lambda_2$. This ambiguity can however be ignored in the present construction, since there exists a Legendrian isotopy of $W^1$ to itself that reverses its orientation.
 
\qed

\subsection{Proof of Theorem \ref{thm:filproductlooseproducttwist}}

\emph{Part (i):} To check the exact Lagrangian condition we compute the pull-back
$$(\iota_1 \boxtimes \iota_2)^{\ast}e^t(dz+\eta_1+\eta_2)=\iota_1^*(e^{t}(dz_1+\eta_1))+e^{t\circ \iota_1}\iota_2^*(dz_2+\eta_2)=\iota_1^*(e^{t}(dz_1+\eta_1))$$
which clearly is exact with a globally constant primitive for $t \ll 0$ by the assumption that $L$ is exact.

\emph{Part (ii):} Let $R_{abc} \subset (P_1^{2n_1} \times \R,dz_1+\eta_1)$ be a stabilised (in the case $n_1=1$) or loose (in the case $n_1 \ge 2)$ neighbourhood of $\Lambda_1$. The standard Legendrian neighbourhood theorem \cite{Geiges} allows us to find a Darboux neighbourhood $U \subset (P_2^{2n_2} \times \R,dz_2+\eta_2)$ near any point of $\Lambda_2$ which is strictly contactomorphic to $\left([-\epsilon,\epsilon]^{2n_2+1},dz-\sum_i y_idx_i\right)$ for some sufficiently small $\epsilon>0$, in which $\Lambda_2$ moreover is equal to $\{y_i=0=z\}$. 

After rescaling the second factor by the positive Liouville flow $(\varphi_{\eta_2}^{t},e^{t}\cdot),$ we can in addition make the assumption that $U$ is symplectomorphic to the huge Darboux neighbourhood $\left([-e^{t/2}\epsilon,e^{t/2}\epsilon]^{2n_2+1},dz-\sum_i y_idx_i\right)$ for an arbitrarily large $ t \gg 0.$ (Here we use the assumption that $\Lambda_1 < \Lambda_2$ in order to infer that the rescaling induces a Legendrian isotopy of the product $\Lambda_1 \boxtimes \Lambda_2$.) After such a rescaling the product
$$ (R_{abc}\times (U \cap \{z=0\}), (R_{abc} \cap \Lambda_1) \times (U \cap \Lambda_2)) \subset (P_1 \times P_2 \times \R,\Lambda_1 \boxtimes \Lambda_2),$$
can be readily seen to become contactomorphic to a loose neighbourhood of $\Lambda_1\boxtimes \Lambda_2$ as sought.\qed

\begin{remark}
Theorem \ref{thm:filproductlooseproducttwist} in particular implies that
\begin{itemize}
\item if $\Lambda_1<\Lambda_2$, $\Lambda_1$ is loose and $\Lambda_2$ is fillable, then $\Lambda_1\boxtimes \Lambda_2$ is loose;
\item if $\Lambda_1<\Lambda_2$, $\Lambda_1$ is fillable and $\Lambda_2$ is loose, then $\Lambda_1\boxtimes \Lambda_2$ is fillable.
\end{itemize}
In other words, without any extra assumptions on the sizes of Reeb chords, Legendrian product construction neither preserve looseness, nor fillability of the components.
\end{remark}

\section{Structural results of DGAs of twist spuns}
\label{sec:spundga}

In this section we restrict our attention to  Legendrian tori inside $\R^5$ that arise as twist $S^1$-spuns of families of Legendrian knots inside the standard contact vector space $\R^3$. The Chekanov--Eliashberg algebra of such a twist spun was computed by Ekholm--K{\'a}lm{\'a}n in \cite[Theorem 1.1]{IsotopiesTori} in terms of the DGA of the knot and the DGA-endomorphism induced by the loop of knots, and we begin by recalling this result. This is the crucial ingredient in the proof of the below structural result Theorem \ref{thm:Kunnethformula} for the bilinearised Legendrian contact homology of such tori; see \cite{Bilinearised} for the definition of bilinearisation.

Denote by $(\mathcal{A}(\Lambda_0),\partial)$ the Chekanov--Eliashberg algebra of the knot where we use coefficients $\F$,
and let $$\Phi \colon (\mathcal{A}(\Lambda_0),\partial) \to (\mathcal{A}(\Lambda_0),\partial)$$ be the unital DGA quasi-isomorphism induced by the loop of Legendrians. The Legendrian torus twist spun $\Sigma\{\Lambda_\theta\} \subset \R^5$ has a Chekanov--Eliashberg algebra $(\mathcal{A}(\Sigma\{\Lambda_\theta\}),D)$ which after a suitable perturbation is of the following form.

{\bf Generators:} For each Reeb chord generator $x \in \mathcal{A}(\Lambda_0)$ there are two generators $x$ and $\hat{x}$ of $(\mathcal{A}(\Sigma\{\Lambda_\theta\}),D)$ where the degree of $x$ agrees in both algebras, while $|\hat{x}|=|x|+1.$

{\bf Differential:} For any $x \in \mathcal{A}(\Lambda_0)$ we have $D(x)=\partial(x),$ while for $\hat{x}$ we have
$$ D(\hat x)=\Phi(x)-x+\sum_{\mathbf{b}c\mathbf{d}} \langle \partial(x),\mathbf{b}c\mathbf{d} \rangle \Phi(\mathbf{b})\hat{c}\mathbf{d}.$$

Note that, even when $\Lambda_0$ has rotation number zero, it could be the case that the isotopy $\Lambda_\theta$ induces a shift of Maslov potentials; this is precisely the case when the Maslov class of the corresponding twist spun torus becomes non-trivial.

The above structure of the Chekanov--Eliashberg algebra of the twist spun immediately implies that the DGA of the knot sits included inside it as the sub-DGA generated by the ``generators without hats''.  It is not difficult to show that this also is the case when coefficients in the group ring of $H_1$ is used. The following is the main structural result that we need for the DGA and bilinearised Legendrian contact homology of a twist spun:
\begin{thm}
\label{thm:Kunnethformula}
 
Let $\Lambda_{\theta} \subset \R^3$, $\theta\in S^1$, be a loop of Legendrian knots. There exists an inclusion $\mathcal{A}_{\F[\mu^{\pm 1}]}(\Lambda_0) \subset \mathcal{A}_{\F[\mu^{\pm 1},\lambda^{\pm1}]}(\Sigma\{\Lambda_{\theta}\})$ of unital DGAs that extends the natural inclusion $\F[\mu^{\pm 1}] \subset \F[\mu^{\pm 1},\lambda^{\pm1}]$ of the group ring coefficients (for suitable identifications of $H_1$). Now assume that $\mathcal{A}_{\F[\mu^{\pm 1},\lambda^{\pm1}]}(\Sigma\{\Lambda_{\theta}\})$ admits an augmentation $\tilde{\varepsilon}.$
\begin{enumerate}
\item The augmentation $\tilde{\varepsilon}$ pulls back to an augmentation $\varepsilon$ of $\mathcal{A}_{\F[\mu^{\pm 1}]}(\Lambda_0)$ under the above inclusion of DGAs, which gives rise to an inclusion
$$ LCC^{\varepsilon}_{\ast}(\Lambda_0) \subset LCC^{\tilde{\varepsilon}}_{\ast}(\Sigma\{\Lambda_{\theta}\})$$
of linearised complexes. Furthermore, the linearised complex of the torus satisfies
$$ LCC^{\tilde{\varepsilon}}_{\ast}(\Sigma\{\Lambda_{\theta}\}) = \OP{Cone}(\psi)$$
for some (graded) chain-map
$$ \psi \colon LCC^{\varepsilon \circ \Phi,\varepsilon}_{\ast}(\Lambda_0) \to LCC^{\varepsilon}_{\ast}(\Lambda_0)$$
between bilinearised homology complexes (when the augmentation is graded);
\item The augmentation $\varepsilon$ from Part (1) and $\varepsilon \circ \Phi$, where $\Phi \colon \mathcal{A}(\Lambda_0) \to \mathcal{A}(\Lambda_0)$ denotes the quasi-isomorphism of Chekanov--Eliashberg algebras induced by the loop $\Lambda_\theta$, are DG-homotopic. In particular,
$$LCH^{\varepsilon \circ \Phi,\varepsilon}_{\ast}(\Lambda_0) \cong LCH^{\varepsilon}_{\ast}(\Lambda_0)$$
is satisfied; and
\item The unital DGA quasi-isomorphism $\Phi \colon \mathcal{A}(\Lambda_0) \to \mathcal{A}(\Lambda_0)$ of Chekanov--Eliashberg algebras induced by the loop $\Lambda_\theta$ extends to a unital DGA quasi-isomorphism $\tilde{\Phi}$, i.e.
$$
\xymatrix{
\mathcal{A}(\Sigma\{\Lambda_{\theta}\}) \ar[r]^{\tilde{\Phi}}& \mathcal{A}(\Sigma\{\Lambda_{\theta}\})\\
\mathcal{A}(\Lambda_0) \ar@{^{(}->}[u] \ar[r]_{\Phi} & \ar@{^{(}->}[u] \mathcal{A}(\Lambda_0) }
$$
under the above inclusion of DGAs.
\end{enumerate}
\end{thm}
\begin{remark}
\begin{itemize}
\item A related result in the case of generating family homology appeared in the work of Sabloff--Sullivan \cite[Propositions 5.4 and 5.5]{FamiliesLegendrianGeneratingFamilies}.
\item The map $\psi$ vanishes when $ \Lambda_\theta \equiv \Lambda_0$ is the constant family, which gives back the K\"{u}nneth formula for the ordinary $S^1$-spun.
\item For untwisted spherical spuns in arbitrary dimension, for augmentations that are induced by spuns of exact Lagrangian fillings, a K\"{u}nneth-type formula was established by Chantraine, Ghiggini and the authors in \cite{Floer_Conc}.
\end{itemize}
\end{remark}
\begin{proof}[Proof of Theorem \ref{thm:Kunnethformula}]
When coefficients are taken in $\F$ as opposed to the group ring of $H_1$, the existence of the inclusion of the DGA is immediate from the above presentation of the Chekanov--Eliashberg algebra of a twist spun. The refined result with group ring coefficients can also be seen to follow by the same analysis from \cite{IsotopiesTori}, as we now show.

We briefly recall the geometric correspondence between the discs that contribute to $D|_{\mathcal{A}(\Lambda_0)}$ and the discs that contribute to $\partial.$ It suffices to consider the Legendrian suspension inside $\R^3 \times T^*S^1$ instead of the twist spun inside $\R^5$, since the Chekanov--Eliashberg algebras are the same.

One starts by choosing a Morse function on $S^1$ with precisely two critical points in order to perturb the Legendrian; we assume that the minimum is located at $\theta=0$. After some care has been taken, the twist spun may be assumed to intersect the hypersurface
$$\R^3 \times \{\theta=0\} \subset \R^3 \times T^*S^1$$
transversely in the knot
$$\Lambda_0 \times \{(\theta,p)=0\} \subset \R^3 \times T^*S^1.$$
Moreover, there is an induced grading-preserving correspondence between the Reeb chord generators of $\mathcal{A}_{\F[\mu^{\pm 1},\lambda^{\pm1}]}(\Sigma\{\Lambda_{\theta}\})$ contained inside this hypersurface, and the Reeb chord generators of $\mathcal{A}(\Lambda_0)$; these are indeed the generators of the sought sub-DGA. The key point is that, for suitable choices, the pseudoholomorphic discs that contribute to $D(x)$ for these generators all live inside the symplectic hypersurface
$$\{(\theta,p)=0\} \subset \R^2_{xy} \times T^*S^1$$
and can be identified with the discs that contribute to $\partial(x)$. Using the same geometric correspondence of discs, it is now straight-forward to also obtain the identification of first homology classes of the boundaries of the disc, which are responsible for giving the contributions to the homology coefficients.

(1): The inclusion of linearised complexes is immediate from the existence of the unital inclusion of DGAs established above.

We proceed to exhibit the cone structure. Yet another direct consequence of the above expression for the DGA of a twist spun is that the quotient complex $LCC^{\tilde\varepsilon}_\ast(\Sigma\{\Lambda_\theta\})/LCC^{\varepsilon}_\ast(\Lambda_0)$ is isomorphic to the bilinearised complex $LCC^{\varepsilon \circ \Phi,\varepsilon}_\ast(\Lambda_0)$. (See \cite{Bilinearised} for the definition of bilinearised Legendrian contact homology.) The structure of a mapping cone is now also evident.

(2): The equations $\tilde{\varepsilon}\circ D(\hat{x})=0$, which are satisfied since $\tilde{\varepsilon}$ is an augmentation, together with the assumption that $\tilde{\varepsilon}(x)=\varepsilon(x)$ holds for the generators $x \in \mathcal{A}(\Lambda_0)$, are precisely the equations that give the sought DG-homotopy between $\epsilon$ and $\epsilon \circ \Phi$; see \cite[Lemma 3.14]{Ekhoka} for the definition of this relation. Theorem \ref{thm:dghomotopy} now implies that
$$LCH^{\varepsilon \circ \Phi,\varepsilon}_\ast(\Lambda_0) \cong LCH^{\varepsilon,\varepsilon}_\ast(\Lambda_0)=LCH^\varepsilon_\ast(\Lambda_0)$$
as sought.

(3): The twist spun can be rotated to yield an $S^1$-family of tori. More precisely, the loop of Legendrians $\{\Lambda_{\theta}\}_\theta$ sits inside an $S^1$-family of loops $\{\Lambda_{\tau+\theta}\}_\theta$ that is parametrised by $\tau \in S^1.$ We thus get a loop of the corresponding twist spuns that depends on the parameter $\tau \in S^1$.

The sought DGA quasi isomorphism $\tilde{\Phi}$ is the one associated to this Legendrian isotopy of twist spuns, as produced by the invariance proof of the Chekanov--Eliashberg algebra in \cite{LegendrianContactPxR,IsotopiesTori}. If we perturb each twist spun in the loop by the Morse function on $S^1$ above, so that $\mathcal{A}(\Lambda_\tau)$ is included as a sub-algebra generated by the chords in the hypersurface $\{\theta=0\}$ over the minimum for all $\tau \in S^1$, it readily follows from the analysis in the invariance proof that $\tilde{\Phi}$ restricts to $\Phi$ on this subalgebra. 
\end{proof}

Leverson \cite{Leverson} has shown that any graded augmentation of a knot must map the  coefficient $\mu \in \F[\mu^{\pm1}]=\F[H_1(\Lambda)]$ to  $\mu \mapsto -1$ (for a suitable choice of spin structure). Recall that the augmentation variety of a Legendrian torus is the algebraic closure of those points in $(\C^*)^2$, thought of as $\C$-algebra maps $\C[\mu^{\pm 1},\lambda^{\pm1}] \to \C,$ which extend to graded augmentations of its Chekanov--Eliashberg algebra. From Part (1) of Theorem \ref{thm:Kunnethformula} above we thus conclude that:
\begin{cor}
\label{cor:augvarytwistspun}
The augmentation variety of a twist-spun is contained inside the affine line $$\{\mu=-1\} \subset (\C^*)^2$$ after a suitable identification $H_1(\T^2) \cong \Z\mu\oplus \Z\lambda$ and choice of spin structure on the torus.
\end{cor}

\section{Proof of Theorem \ref{thm:productnotspunnottwistspun}}
\label{sec:productnottwist}
In this section we will analyse some of the simplest possible examples of Legendrian products $\Lambda_1 \boxtimes \Lambda_2$ of two Legendrian knots $\Lambda_1, \Lambda_2 \subset \R^3$ for which neither $\Lambda_1 < \Lambda_2$ nor $\Lambda_1 > \Lambda_2$ is satisfied. In particular, we will consider it from the point of view of fillability and also perform some partial computations of their Chekanov--Eliashberg algebras. Since Theorem \ref{thm:normalbundle} does not apply to products of this kind, it is a priori not clear whether such a product is a twist spun or not. Indeed, we will exhibit examples which are not Legendrian isotopic to any twist spun of a knot.

\subsection{The family of examples}

Consider the standard Legendrian unknot $\Lambda_1 \subset \R^3$ with a single Reeb chord $a$ of length $\ell(a)=1$ and grading $|a|=1$. We then proceed to construct the second Legendrian factor $\Lambda_2^{2r}$, which comes in a family that depends on the parameter $r = 0,1,2,3,\ldots$.

Start by considering a Legendrian unlink $\Lambda^{2r}_-$ that consists of two Legendrian unknots, each of $\OP{rot}=r$ and ${\tt tb}=-1-r$, where the link is symmetric under the rotation $(x,y,z) \mapsto (-x,-y,z)$ (or, equivalently, under a reflection of the front), and moreover disjoint from the hyperplane $x=0$. We pick representatives of this link for which each component of $\Lambda^{2r}_-$, indexed by $i=1,2$, has precisely a number $1+r$ of transverse chords: a number $r$ of chords $c^i_1,\ldots,c^i_r$ of length $1+\epsilon$ and a single chord $d^i$ of length $1+r$. The case $r=1$ is depicted on the right-hand side of Figure \ref{fig:surgery}. In general, one can start with the standard unknot $W^1$ with a single chord of length $1+r$, and then perform a number $r$ of stabilisations.

We then perform a cusp connect sum, as defined in \cite{ConnectedSum} by Etnyre--Honda, between the union $\Lambda^{2r}_-$ of two Legendrian unknots by connecting two cusp edges that face each other by a horizontal Legendrian arc $D$ on which the $y$-coordinate is constant; again see Figure \ref{fig:surgery} for the case $r=1$. We denote the resulting Legendrian unknot by $\Lambda_2^{2r}$. Note that the result of the connected sum is a Legendrian unknot that satisfies $\OP{rot}(\Lambda_2^{2r})=0$. (This unknot is obviously stabilised whenever $r>0$.)

 After the surgery the Reeb chords become absolutely graded; one computes $|c^i_j|=1$ and $|d^i|=2r+1$. The new Reeb chord $b$ produced by the cusp connected sum is then seen to be of degree $2r \in \Z$; we may assume that its length is arbitrarily small and equal to $\epsilon>0$. The Legendrian $\Lambda^2_2$ resulting from the connected sum can be seen on the right-hand side of Figure \ref{fig:NonspunProduct}.

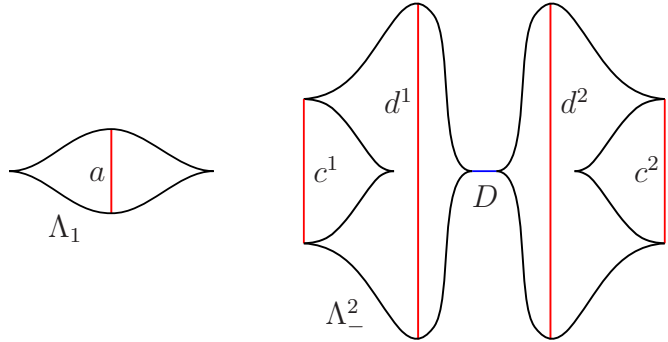
\begin{figure}[t]
\centering \labellist \pinlabel $\Lambda_1$ at 35 70 \pinlabel $a$ at 55 104
\pinlabel $D$ at 297 90 \pinlabel $\Lambda^2_-$ at 210 15
\pinlabel $d^1$ at 243 150
\pinlabel $d^2$ at 353 150
\pinlabel $c^1$ at 198 107
\pinlabel $c^2$ at 398 107
\endlabellist
\includegraphics[height=4.5cm]{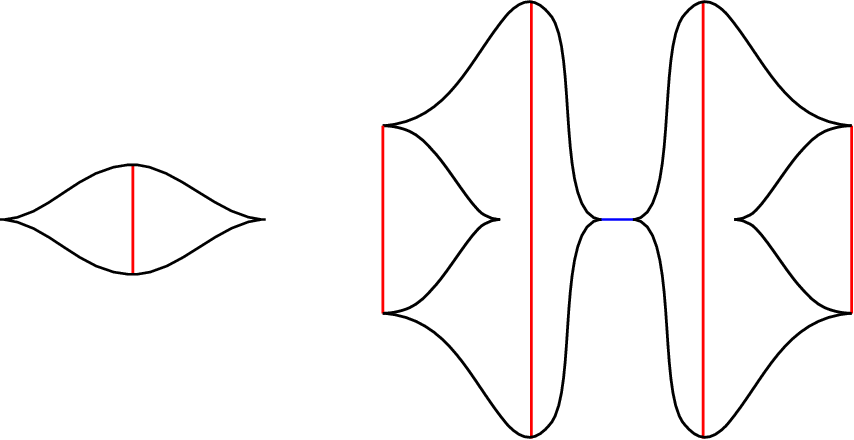}
\caption{Left: the front projection of $\Lambda_1$ with $\ell(a)=1$ and $|a|=1$. Right: the front
projection of the two Legendrian unknots $\Lambda^2_-$ together with the surgery disc $D$.} \label{fig:surgery}
\end{figure}

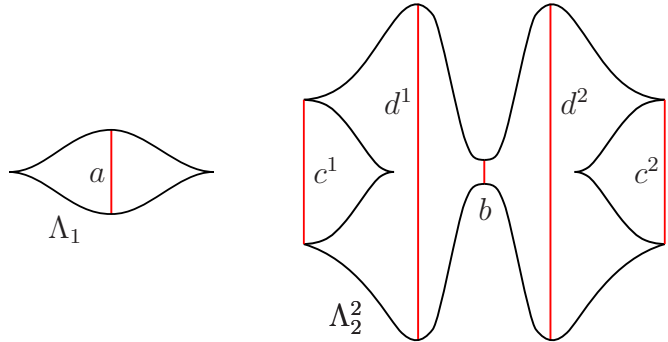
\begin{figure}[t]
\centering \labellist \pinlabel $\Lambda_1$ at 35 70 \pinlabel $a$ at 55 104
\pinlabel $b$ at 297 83 \pinlabel $\Lambda^2_2$ at 210 15
\pinlabel $d^1$ at 243 150
\pinlabel $d^2$ at 353 150
\pinlabel $c^1$ at 198 107
\pinlabel $c^2$ at 398 107
\pinlabel $\Lambda^2_2$ at 210 15
\endlabellist
\includegraphics[height=4.5cm]{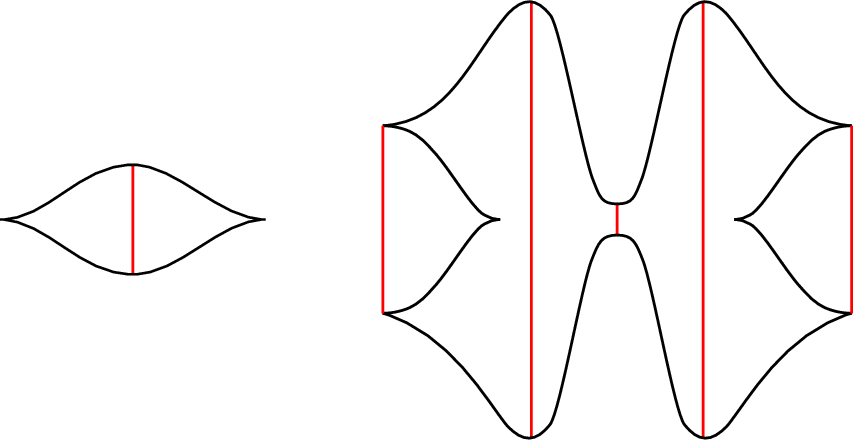}
\caption{Left: the front projection of $\Lambda_1$ with $\ell(a)=1$ and $|a|=1$. Right: the front
projection of $\Lambda^2_2$ with $\ell(b)<1$ and $|b|=2$. All other Reeb chords have lengths greater than one.} \label{fig:NonspunProduct}
\end{figure}

\subsection{Preliminary results}

We do not compute the full DGA of the products under consideration. However, since they are exact fillable as shown below, the quasi-isomorphism class of  the completions of their DGAs with respect to the word-length filtration  can be determined by the topology of the fillings by the work \cite{EkholmLekili} of Ekholm--Lekili. Instead of computing the entire DGA, we will make a much more modest computation that largely is based upon the lengths and degrees of the Reeb chords on the product.
\begin{prop}
 \label{prp:Degrees}
The Legendrian product $\Lambda_1 \boxtimes \Lambda_2^{2r}$ is Legendrian isotopic to a representative whose Reeb chords all are transverse and given as follows:
\begin{itemize}
\item $a$ and $A$ in degrees $|a|=|A|-1=1$ of lengths $\ell(a)=\ell(A)-\delta=1$;
\item $b$ and $B$ in degrees $|b|=|B|-1=2r$ of lengths $\ell(b)=\ell(B)-\delta=\epsilon;$
\item $c_{a+b}$ in degree $|c_{a+b}|=|a|+|b|+1=2+2r$ of length $\ell(c_{a+b})=\ell(a)+\ell(b)=1+\epsilon$; and
\item $c_{a-b}$ in degree $|c_{a-b}|=|a|-|b|=1-2r$ of length $\ell(c_{a-b})=\ell(a)-\ell(b)=1-\epsilon$;
\end{itemize}
where $\epsilon>0$ is small but fixed, and $0<\delta<\epsilon$ is sufficiently small.
\end{prop}
\begin{proof}
Recall that Lambert--Cole computed the degrees of the Reeb chords of the product in \cite{LegendrianProducts}; they can be determined by the chords on either component. For our product $\Lambda_1 \boxtimes \Lambda_2^{2r}$ we deduce that the set of Reeb chords after a small Legendrian perturbation are as follows:
\begin{itemize}
\item $a$ and $A$ in degrees $|a|=|A|-1=1$ of lengths $\ell(a)=\ell(A)-\delta=1$;
\item $b$ and $B$ in degrees $|b|=|B|-1=2r$ of lengths $\ell(b)=\ell(B)-\delta=\epsilon$;
\item $c^i_j$ and $C^i_j$ in degrees $|c^i_j|=|C^i_j|-1=1$ of lengths $\ell(c^i_j)=\ell(C^i_j)-\delta=1+\epsilon$ for $i=1,2$ and $j=1,\ldots,r$;
\item $d^i$ and $D^i$ in degrees $|d^i|=|D^i|-1=2r+1$ of lengths $\ell(d^i)=\ell(D^i)-\delta=1+r$ for $i=1,2$;
\item $c_{a+b}$ in degree $|c_{a+b}|=|a|+|b|+1=2+2r$ of length $\ell(c_{a+b})=\ell(a)+\ell(b)=1+\epsilon$;
\item $c_{a-b}$ in degree $|c_{a-b}|=|a|-|b|=1-2r$ of length $\ell(c_{a-b})=\ell(a)-\ell(b)=1-\epsilon$;
\item $c_{c^i_j+a}$ in degree $|c_{c^i_j+a}|=|c^i_j|+|a|+1=3$ of length $\ell(c_{c^i_j+a})=\ell(c^i_j)+\ell(a)=2+\epsilon$ for $i=1,2$ and $j=1,\ldots,r$;
\item $c_{c^i_j-a}$ in degree $|c_{c^i_j-a}|=|c^i_j|-|a|=0$ of length $\ell(c_{c^i_j-a})=\ell(c^i_j)-\ell(a)=\epsilon$ for $i=1,2$ and $j=1,\ldots,r$;
\item $c_{d^i+a}$ in degree $|c_{d^i+a}|=|d^i|+|a|+1=2r+2$ of length $\ell(c_{d^i+a})=\ell(d^i)+\ell(a)=2r+2$ for $i=1,2$; and
\item $c_{d^i-a}$ in degree $|c_{d^i-a}|=|d^i|-|a|=2r+1$ of length $\ell(c_{d^i-a})=\ell(d^i)-\ell(a)=2r$ for $i=1,2$.
\end{itemize}
Here $\epsilon>0$ is small but fixed, and $0<\delta<\epsilon$ is sufficiently small. All Reeb chords are moreover transversely cut out.

Except for the chords appearing under the first four bullet points, the Lagrangian projection of the chords above coincide with the Cartesian product of double points of the Lagrangian projections $\Pi_{\R^2}(\Lambda_1) \times \Pi_{\R^2}(\Lambda^{2r}_2)$ of the two factors. These products of chords are automatically transverse on the Legendrian product, given that the chords on each of the two factors are transverse.

On the other hand, the pair of chords appearing in each of the first four bullet points arise from a generic perturbation of an $S^1$-Bott family of chords of the form $x \times \Lambda^{2r}_2$ or $\Lambda_1 \times y$, where $x$ and $y$ are chords on each factor. The perturbation can be determined by a choice of Morse function on the $S^1$-family of chords, and the resulting generic chords appear arbitrarily close to the critical points of this Morse function; we chose the number of critical points to be equal to two.

When perturbing the family $a \times \Lambda^{2r}_2$ of chords we make an additional choice of Morse function, so that the resulting generic chords $a$ and $A$ on the perturbed product both become contained inside the hypersurface $\{x_2=0\} \subset \R^2_{x_1y_1} \times \R^2_{x_2y_2} \times \R_z$. Recall that $\{x_2=0\}$ is the axis of symmetry of the Legendrian knot $\Lambda^{2r}_2$.

It now suffices to show the following: all chords on the above
perturbation of the product, except those contained near the
hypersurface $\{x_2=0\}$, can be removed by a suitable Legendrian isotopy
which fixes a neighbourhood of the latter hypersurface. In the
remainder of the proof we argue how to construct such a Legendrian isotopy.

For any smooth real-valued function $\phi \colon \Lambda^{2r}_2 \to \R$ which satisfies $d\phi=0$ in a neighbourhood of all critical points of the function $x_2$, we can deform the Legendrian product $\Lambda_1 \boxtimes \Lambda^{2r}_2$ in the following manner. Consider the family of Legendrian tori $\Lambda_1 \boxtimes_t \Lambda^{2r}_2$ smoothly depending on the parameter $t \in \R$, uniquely determined by the front projections
$$(u_1,u_2) \mapsto (x_1(\iota_1(u_1))+t\cdot\phi(\iota_2(u_2)),x_2(\iota_2(u_2)),z_1(\iota_1(u_1))+z_2(\iota_2(u_2))) \in \R^2_{x_1x_2} \times \R_z,$$
where $\Lambda_1 \boxtimes_0 \Lambda^{2r}_2=\Lambda_1 \boxtimes \Lambda^{2r}_2$ thus is the original product. In order to reach our goal, we can take the function $\phi$ to vanish in a neighbourhood of the hypersurface $\{x_2=0\}$ (the axis of symmetry of $\Lambda^{2r}_2$) while it is equal to the $z_2$-coordinate near the starting and endpoints of each Reeb chord of $\Lambda^{2r}_2$ that is contained in the complement of the same hypersurface.

We start with the observation that all chords on the above Legendrians $\Lambda_1 \boxtimes_t \Lambda^{2r_r}_2$ that are contained inside the region $\{x_2 \neq 0\}$ live inside subsets of the form $\R^2_{x_1y_1} \times \{(x_2,y_2)=u \} \times \R_z$, where $u \in \Pi_{\R^2}(\Lambda^{2r}_2) \cap \{x_2 \neq 0\}$ is a double point of the Lagrangian projection of the second factor. The latter double points moreover correspond to chords of length strictly greater than $1=\ell(a)$.

From the previous paragraph, together with our choice of $\phi$, it is not difficult to see that there are no Reeb chords on $\Lambda_1 \boxtimes_t \Lambda^{2r}_2$ away from the hypersurface $\{x_2=0\}$ whenever $t \gg 0$ is sufficiently large.

It remains to show that varying $t$ yields a Legendrian isotopy. To that end, the crucial property that we need is the fact that the link $\Lambda_1 \cup \Lambda_1' \subset \R^3_{x_1y_1z_1}$ is an \emph{embedded} Legendrian link whenever $\Lambda_1'$ is obtained from $\Lambda_1$ by a translation in the $(x_1,z_1)$-plane which translates the $z_1$--coordinate by a distance strictly greater than $1=\ell(a)$. By this property, together with the fact that the Reeb chords on $\Lambda^{2r}_2$ away from $\{x_2=0\}$ have length strictly greater than $1=\ell(a)$, all chords on the above family of Legendrians can finally be seen to have a positive length for all $t \in \R$.
\end{proof}

Let $L$ denote the compact three-dimensional manifold with boundary $\partial L \cong \T^2$ obtained from $S^2 \times [0,1]$ by two oriented one-handle attachments.
\begin{prop}
\label{prp:filling}
The Legendrians $\Lambda_1 \boxtimes \Lambda_2^{2r}$ admit exact Lagrangian fillings diffeomorphic to $L$, whose Maslov class moreover vanishes when $r=0$.
\end{prop}
\begin{proof}
The standard Legendrian unknot $\Lambda_1$ admits an exact filling by a Lagrangian disc. Since $\Lambda_1< \Lambda^{2r}_-$, Part (i) of Theorem \ref{thm:filproductlooseproducttwist} implies that $\Lambda_1 \boxtimes \Lambda^{2r}_-$ admits an exact Lagrangian filling consisting of a disjoint union of two solid tori. The Maslov class of this solid torus vanishes if and only if $r=0$.

The handle-attachment cobordism $V$ from $\Lambda^{2r}_-$ to $\Lambda^{2r}_+$ which corresponds to the cusp connect sum, see \cite{Ambient} for the construction, can be taken to be contained in an arbitrarily small neighbourhood of the ambient surgery disc $D$ shown in Figure \ref{fig:surgery}. Part (i) of Theorem \ref{thm:filproductlooseproducttwist} then again shows that the product $\Lambda_1 \boxtimes V$ is an embedded exact Lagrangian cobordism from $\Lambda_1 \boxtimes \Lambda^{2r}_-$ to $\Lambda_1 \boxtimes \Lambda^{2r}_2.$

The sought exact filling is then given by the concatenation of the filling of $\Lambda_1 \boxtimes \Lambda^{2r}_-$ and the cobordism $\Lambda_1 \boxtimes V$.
\end{proof}

\begin{prop}
\label{prp:UniqueAug}
When $r \ge 0,$ the Chekanov--Eliashberg algebra of the Legendrian $\Lambda_1 \boxtimes \Lambda_2^{2r}$ has a graded augmentation. When $r>0$, this augmentation is moreover unique up to DGA homotopy and, for any pair of graded augmentations, it is the case that 
\begin{equation}
\label{eq:negdegrees}
LCH^{\tilde{\varepsilon}_1,\tilde{\varepsilon}_2}_{\ast}(\Lambda_1 \boxtimes \Lambda_2^{2r})=\begin{cases}
\F, & \ast = 1-2r;\\
0, & \text{otherwise},
\end{cases}
\end{equation}
 for nonpositive degrees $\ast \leq 0$.
\end{prop}
\begin{proof}
The exact Lagrangian filling provided by Proposition \ref{prp:filling} implies the existence of a (possibly ungraded) augmentation; see Theorem \ref{thm:filling} proven in \cite{RationalSFT,RationalSFT2}. In the case $r=0$, due to the fact that the filling has vanishing Maslov class, this augmentation is moreover graded.

In the case $r>0$ we need to use the degree and length computations from Proposition \ref{prp:Degrees}. After the isotopy in that proposition, we have no Reeb chord generators in degree zero. The uniqueness of the DG-homotopy class of graded augmentations now follows if we just can show that a single graded augmentation exists.

For the existence, we need argue that the canonical algebra map to $\F$ is an augmentation. A sufficient condition for this is that $\partial(x)=0$ is satisfied for all Reeb chord generators $x$ of degree $|x|=1$ (in particular this implies no nonzero constant terms). Indeed, by comparing lengths and degrees in Proposition \ref{prp:Degrees} it is the case that $\partial(x)$ only can contain constant terms. In fact, this constant term must be zero as sought, since otherwise the DGA would be acyclic, which is in contradiction with the existence of the (possibly ungraded) augmentation induced by the exact Lagrangian filling from Proposition \ref{prp:filling}.

Finally, the computation of the bilinearised Legendrian contact homology is a direct consequence of the degree computations in Proposition \ref{prp:Degrees} and the invariance of these homology groups.
\end{proof}

\subsection{The product is not a twist spun (Proof of Theorem \ref{thm:productnotspunnottwistspun})}
We show that the Legendrians $\Lambda_1 \boxtimes \Lambda_2^{2r}$ are not Legendrian isotopic to twist spuns whenever $r \ge 1.$ We believe that the statement is true also when $r=0,$ but in that case a more refined computation of the DGA would be necessary.

We argue by contradiction and assume that $\Lambda_1 \boxtimes \Lambda_2^{2r}$ is Legendrian isotopic to a twist spun. Part (2) of Theorem \ref{thm:Kunnethformula} implies that
\begin{equation}
\label{eq:lch}
LCH^{\tilde{\varepsilon}}_{*}(\Lambda_1 \boxtimes \Lambda_2^{2r}) \cong H(\OP{Cone}(\psi)),
\end{equation}
for a chain map $\psi$ that induces the graded automorphism
$$ [\psi]_* \colon LCH_*^\varepsilon(\Lambda) \to LCH_*^\varepsilon(\Lambda) $$
in homology. Observe that the homology group on the left-hand side of Formula \eqref{eq:lch} was computed earlier in non-positive degrees; by Formula \eqref{eq:negdegrees} in Proposition \ref{prp:UniqueAug} it is equal to $\F$ in degree $*=1-2r$ and vanishes in other non-positive degrees.

To conclude, the above cone structure induces a long exact sequence in homology which thus takes the form
$$ 0 \to LCH^\varepsilon_{1-2r}(\Lambda) \xrightarrow{ \delta_{1-2r}} LCH_{1-2r}^\varepsilon(\Lambda) \to H_{1-2r}(\OP{Cone}(\psi)) \to LCH_{-2r}^\varepsilon(\Lambda) \xrightarrow{ \delta_{-2r}} LCH_{-2r}^\varepsilon(\Lambda) \to 0.$$
Exactness implies that the connecting homomorphisms $\delta_{1-2r}=[\psi]_{1-2r}$ as well as $\delta_{-2r}=[\psi]_{-2r}$ are injective and surjective, respectively. By finite dimensionality they are both isomorphisms. However, since
$$H_{1-2r}(\OP{Cone}(\psi)) \cong \F \neq 0$$
holds, we arrive at a contradiction with the exactness of the above sequence. \qed

\section{Other examples of Legendrians that are not twist spuns}
\label{sec:other}

The two Legendrian tori constructed in \cite{LegendrianBSLagrangianTori}, corresponding to suitable threefold covers of the Clifford and Chekanov torus, are here shown to not be Legendrian isotopic to twist spuns. This is done by mere considerations of their augmentation varieties, while taking the structural result for the DGA of a twist spun from Part (1) of Theorem \ref{thm:Kunnethformula} into account.

\begin{figure}[h!]
\centering
\includegraphics[height=4cm]{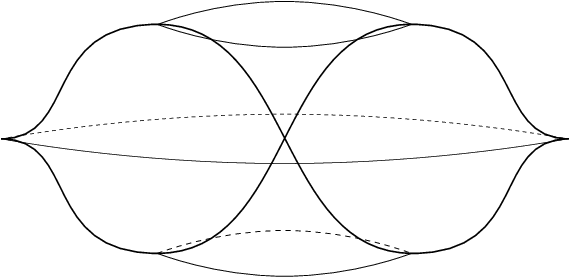}
\caption{Front projection of $\Lambda_{\OP{Cl}}$.}
\label{fig:Clif}
\end{figure}
We consider a conical special Lagrangian inside $\R^6$, whose intersection with the standard
contact sphere $S^5$ is a Legendrian torus which projects to $\C P^2$ as a threefold cover of
the monotone Clifford torus (this is the Legendrian link of the so-called Harvey--Lawson cone). After a Legendrian isotopy into a small contact Darboux ball, we get the Legendrian $\Lambda_{\OP{\OP{Cl}}}\subset J^{1}(\R)$ with the front projection shown in Figure \ref{fig:Clif}.
The computation of it appeared in \cite{KnottedLegendrianSurface} and \cite{LegendrianBSLagrangianTori}:

\begin{prop}\label{prop:augmvarcliftor}
For the Lie group spin structure and suitable choices of capping paths and basis $\{\mu,\lambda\}$ of $H_1(\Lambda_{\OP{Cl}})$,
the augmentation variety of $\Lambda_{\OP{Cl}}$ is equal to the one-dimensional complex pair of pants
$$\OP{Sp}(\C[\mu^{\pm1},\lambda^{\pm1}]/\langle 1+\lambda(1+\mu)\rangle).$$
\end{prop}

\begin{figure}[h!]
\centering
\includegraphics[height=4cm]{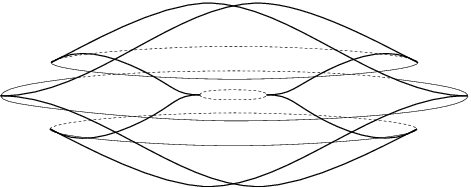}
\caption{Front projection of $\Lambda_{\OP{Ch}}$.}
\label{fig:Ch}
\end{figure}
We consider the Legendrian lift $\Lambda_{\OP{Ch}}$ of the threefold Bohr--Sommerfeld cover of the Chekanov torus placed inside a Darboux ball, the front projection of it is described in Figure \ref{fig:Ch}.
The augmentation variety of $\Lambda_{\OP{Ch}}$ has been computed by the authors in \cite{LegendrianBSLagrangianTori}:

\begin{prop}\label{prop:augmvarchtor}
For the Lie group spin structure and suitable choices of capping paths and basis $\langle\mu,\lambda\rangle$ of $H_1(\Lambda_{\OP{Ch}})$,
the augmentation variety of $\Lambda_{\OP{Ch}}$ is equal to the one-dimensional complex pair of pants
$$\OP{Sp}(\C[\mu^{\pm1},\lambda^{\pm1}]/\langle 1+\lambda(1+\mu)^{2}\rangle).$$
\end{prop}
Since, in particular, the augmentation varieties of $\Lambda_{\OP{\OP{Cl}}}$ and $\Lambda_{\OP{\OP{Ch}}}$ do not contain a component which is a two-punctured sphere, we can use Corollary \ref{cor:augvarytwistspun} to deduce that
\begin{cor}
Neither $\Lambda_{\OP{\OP{Cl}}}$ nor $\Lambda_{\OP{\OP{Ch}}}$ is the twist spun of a Legendrian knot.
\end{cor}

\begin{remark}
By the result of Vianna \cite{Vianna16}, there exists an infinite family of different monotone Lagrangian tori inside $\C P^2$. Since all these tori have superpotentials whose Newton polytopes are nondegenerate triangles by the same author (which means that their zero loci have at least three punctures, see for example \cite[Section 1.5]{AmoebasAlgebraicVariety}), and since their superpotentials cannot have critical value equal to zero by \cite[Theorem 1.6]{Auroux06}, we expect that their threefold Bohr--Sommerfeld covers as constructed in \cite{LegendrianBSLagrangianTori} give an infinite family of Legendrians which are not twist spuns. This expectation is based on \cite[Conjecture 8.1]{LegendrianBSLagrangianTori} formulated by the authors, by which the augmentation variety is a certain covering space of the zero set of the superpotential, together with Proposition \ref{prp:UniqueAug}.
\end{remark}

\begin{remark}
In addition, we expect that none of the Legendrian tori considered in this section is a Legendrian product. However, we currently lack a structural understanding of the DGAs of Legendrian products that would allow us to attack this question.
\end{remark}

\bibliographystyle{plain}
\bibliography{references}

\begin{thebibliography}{10}

\bibitem{Audin-Lalonde-Polterovich}
M.~Audin, F.~Lalonde, and L.~Polterovich.
\newblock Symplectic rigidity: {L}agrangian submanifolds.
\newblock In {\em Holomorphic curves in symplectic geometry}, volume 117 of
  {\em Progr. Math.}, pages 271--321. Birkh\"auser, Basel, 1994.

\bibitem{Auroux06}
D.~Auroux.
\newblock Mirror symmetry and {$T$}-duality in the complement of an
  anticanonical divisor.
\newblock {\em J. G\"{o}kova Geom. Topol. GGT}, 1:51--91, 2007.

\bibitem{Models}
H.~J. Baues and J.-M. Lemaire.
\newblock Minimal models in homotopy theory.
\newblock {\em Math. Ann.}, 225(3):219--242, 1977.

\bibitem{Bilinearised}
F.~Bourgeois and B.~Chantraine.
\newblock Bilinearized {L}egendrian contact homology and the augmentation
  category.
\newblock {\em J. Symplectic Geom.}, 12(3):553--583, 2014.

\bibitem{Floer_Conc}
B.~{Chantraine}, G.~{Dimitroglou Rizell}, P.~{Ghiggini}, and R.~{Golovko}.
\newblock Floer homology and {L}agrangian concordance.
\newblock In {\em Proceedings of the G\"{o}kova Geometry-Topology Conference
  2014}, pages 76--113. G\"{o}kova Geometry/Topology Conference (GGT),
  G\"{o}kova, 2015.

\bibitem{DiffAlg}
Y.~Chekanov.
\newblock Differential algebra of {L}egendrian links.
\newblock {\em Invent. Math.}, 150(3):441--483, 2002.

\bibitem{KnottedLegendrianSurface}
G.~Dimitroglou~Rizell.
\newblock Knotted {L}egendrian surfaces with few {R}eeb chords.
\newblock {\em Algebr. Geom. Topol.}, 11(5):2903--2936, 2011.

\bibitem{Ambient}
G.~Dimitroglou~Rizell.
\newblock Legendrian ambient surgery and {L}egendrian contact homology.
\newblock {\em J. Symplectic Geom.}, 14(3):811--901, 2016.

\bibitem{EstimatingReebChordsCharacteristicAlgebra}
G.~Dimitroglou~Rizell and R.~Golovko.
\newblock Estimating the number of {R}eeb chords using a linear representation
  of the characteristic algebra.
\newblock {\em Algebr. Geom. Topol.}, 15(5):2887--2920, 2015.

\bibitem{LegendrianBSLagrangianTori}
G.~{Dimitroglou Rizell} and R.~{Golovko}.
\newblock {Legendrian submanifolds from Bohr-Sommerfeld covers of monotone
  Lagrangian tori}.
\newblock {\em arXiv e-prints}, January 2019.

\bibitem{RationalSFT}
T.~Ekholm.
\newblock Rational symplectic field theory over {$\mathbf{Z}_2$} for exact
  {L}agrangian cobordisms.
\newblock {\em J. Eur. Math. Soc. (JEMS)}, 10(3):641--704, 2008.

\bibitem{RationalSFT2}
T.~Ekholm.
\newblock Rational {SFT}, linearized {L}egendrian contact homology, and
  {L}agrangian {F}loer cohomology.
\newblock In {\em Perspectives in analysis, geometry, and topology}, volume 296
  of {\em Progr. Math.}, pages 109--145. Birkh\"auser/Springer, New York, 2012.

\bibitem{NonisotopicLegendrianSubmanifolds}
T.~Ekholm, J.~Etnyre, and M.~G. Sullivan.
\newblock Non-isotopic {L}egendrian submanifolds in {$\mathbf{R}^{2n+1}$}.
\newblock {\em J. Differential Geom.}, 71(1):85--128, 2005.

\bibitem{LegendrianContactPxR}
T.~Ekholm, J.~Etnyre, and M.~G. Sullivan.
\newblock Legendrian contact homology in {$P\times\Bbb R$}.
\newblock {\em Trans. Amer. Math. Soc.}, 359(7):3301--3335, 2007.

\bibitem{Ekhoka}
T.~Ekholm, K.~Honda, and T.~K{\'a}lm{\'a}n.
\newblock {L}egendrian knots and exact {L}agrangian cobordisms.
\newblock {\em Journal of the European Mathematical Society},
  18(11):2627--2689, 2016.

\bibitem{IsotopiesTori}
T.~Ekholm and T.~K\'{a}lm\'{a}n.
\newblock Isotopies of {L}egendrian 1-knots and {L}egendrian 2-tori.
\newblock {\em J. Symplectic Geom.}, 6(4):407--460, 2008.

\bibitem{EkholmLekili}
T.~{Ekholm} and Y.~{Lekili}.
\newblock {Duality between Lagrangian and Legendrian invariants}.
\newblock {\em arXiv e-prints}, January 2017.

\bibitem{Eliash}
Y.~Eliashberg.
\newblock Invariants in contact topology.
\newblock In {\em Proceedings of the {I}nternational {C}ongress of
  {M}athematicians, {V}ol. {II} ({B}erlin, 1998)}, number Extra Vol. II, pages
  327--338, 1998.

\bibitem{IntroSFT}
Y.~Eliashberg, A.~Givental, and H.~Hofer.
\newblock Introduction to symplectic field theory.
\newblock {\em Geom. Funct. Anal.}, (Special Volume, Part II):560--673, 2000.
\newblock GAFA 2000 (Tel Aviv, 1999).

\bibitem{hprinciple}
Y.~Eliashberg and N.~Mishachev.
\newblock {\em Introduction to the {$h$}-principle}, volume~48 of {\em Graduate
  Studies in Mathematics}.
\newblock American Mathematical Society, Providence, RI, 2002.

\bibitem{ConnectedSum}
J.~B. Etnyre and K.~Honda.
\newblock On connected sums and {L}egendrian knots.
\newblock {\em Adv. Math.}, 179(1):59--74, 2003.

\bibitem{Geiges}
H.~Geiges.
\newblock {\em An introduction to contact topology}, volume 109 of {\em
  Cambridge Studies in Advanced Mathematics}.
\newblock Cambridge University Press, Cambridge, 2008.

\bibitem{FrontSpinningConstruction}
R.~Golovko.
\newblock A note on the front spinning construction.
\newblock {\em Bull. Lond. Math. Soc.}, 46(2):258--268, 2014.

\bibitem{LegendrianProducts}
P.~{Lambert-Cole}.
\newblock {Legendrian Products}.
\newblock {\em arXiv e-prints}, page arXiv:1301.3700, January 2013.

\bibitem{GeneratingFamilyLC}
P.~Lambert-Cole.
\newblock {\em Invariants of Legendrian products}.
\newblock PhD thesis, Louisiana State University, 2014.

\bibitem{Leverson}
C.~Leverson.
\newblock Augmentations and rulings of {L}egendrian knots.
\newblock {\em J. Symplectic Geom.}, 14(4):1089--1143, 2016.

\bibitem{AmoebasAlgebraicVariety}
G.~Mikhalkin.
\newblock Amoebas of algebraic varieties and tropical geometry.
\newblock In {\em Different faces of geometry}, volume~3 of {\em Int. Math.
  Ser. (N. Y.)}, pages 257--300. Kluwer/Plenum, New York, 2004.

\bibitem{LooseLegendrianEmbeddings}
E.~{Murphy}.
\newblock {Loose Legendrian Embeddings in High Dimensional Contact Manifolds}.
\newblock {\em arXiv e-prints}, January 2012.

\bibitem{AugmentationsSheaves}
L.~{Ng}, D.~{Rutherford}, V.~{Shende}, S.~{Sivek}, and E.~{Zaslow}.
\newblock {Augmentations are Sheaves}.
\newblock {\em arXiv e-prints}, page arXiv:1502.04939, Feb 2015.

\bibitem{FamiliesLegendrianGeneratingFamilies}
J.~M. Sabloff and M.~G. Sullivan.
\newblock Families of {L}egendrian submanifolds via generating families.
\newblock {\em Quantum Topol.}, 7(4):639--668, 2016.

\bibitem{Vianna16}
R.~Vianna.
\newblock Infinitely many exotic monotone {L}agrangian tori in {$\Bbb{CP}^2$}.
\newblock {\em J. Topol.}, 9(2):535--551, 2016.

\end{thebibliography}
\end{document}